\documentclass[9pt]{jdg-p-1-2}

\newtheorem{theorem}{Theorem}[section]
\newtheorem{lemma}[theorem]{Lemma}
\newtheorem{corollary}[theorem]{Corollary}
\newtheorem{proposition}[theorem]{Proposition}

\theoremstyle{definition}

\theoremstyle{remark}
\newtheorem{remark}[theorem]{Remark}

\def\heat{\lf(\frac{\p}{\p t}-\Delta\ri)}

\def\Ric{\text{Ric}}
\def\lf{\left}
\def\ri{\right}

\def\p{\partial}

\def\R{\Bbb R}

\def\Sph{\Bbb S}

\def\Rc{\operatorname{{\bf R}}}

\def\id{\operatorname{id}}
\def\Ric{\operatorname{Ric}}
\def\I{\operatorname{I}}

\def\tr{\operatorname{tr}}

\def\tri{\operatorname{tri}}

\def\I{\operatorname{I}}

\numberwithin{equation}{section}

\begin{document}

\title{On a classification of the gradient shrinking solitons}

\author{Lei Ni}
\address{Department of Mathematics, University of California at San Diego, La Jolla, CA 92093}
%    Current address

\email{lni@math.ucsd.edu}
\author{Nolan  Wallach}
\address{Department of Mathematics, University of California at San Diego, La Jolla, CA 92093}

\email{nwallach@math.ucsd.edu}
\thanks{ The first author's research  was
supported in part by NSF grant DMS-0504792  and an Alfred P. Sloan
Fellowship, USA.
The second author's research was partially supported by an NSF Summer Grant.
}

\date{June 2007}


\begin{abstract} The main purpose of this article is to provide an alternate proof to a result of Perelman on gradient shrinking solitons. In dimension three we also generalize the result by 
removing the $\kappa$-non-collapsing assumption. In high dimension this new method allows us to  prove a classification result on gradient shrinking solitons with vanishing Weyl curvature tensor, which includes the rotationally symmetric ones.
\end{abstract}

\maketitle

\section{Introduction}

In his surgery paper Perelman proved the following statement
\cite{P2}:

\begin{theorem} \label{perelman1} Any $\kappa$-non-collapsed gradient shrinking
 soliton $M^3$ with bounded
positive sectional curvature must be compact.
\end{theorem}

Combining with Hamilton's convergence (or curvature pinching) result
\cite{H82} (see also \cite{Ivey}) one can conclude that $M^3$ must
be isometric to a quotient of $\Sph^3$. The use of such a
result is that it rules out the possible complications caused by the
existence  of noncompact singularity models and implies a
classification of finite time singularities models,  which then
makes surgery procedure possible in the case of dimension three.
More precisely, ancient solutions, which are noncompact in
interesting cases, can be obtained as the Cheeger-Gromov limit of
the sequence of blow-ups, via the compactness result of Hamilton
\cite{H93}, as we approach to the singular time.  The gradient
shrinking solitons arise from the non-collapsed ancient solutions as the blow-down
limits \cite{P1}, at least in the case that the ancient solution has nonnegative
curvature operator.
By ruling out the noncompact shrinking solitons with
positive curvature one can conclude that the shrinking soliton
arisen from the ancient solutions must be cylinder $\Sph^2\times \R$
or its quotient. This provides the phototype for the surgery.
This relation of the gradient shrinking solitons with the Ricci flow
suggests the importance of studying the   noncompact gradient
shrinking solitons.

On the other hand, Perelman's proof, of which one can find a
 detailed exposition in \cite{CZ, KL, MT} (see also  pages 377-386 of \cite{CLN}),
 is geometric and
relies on detailed analysis of the level sets of the potential
function, and more importantly, the Gauss-Bonnet formula for
surfaces. The authors could not adapt Perelman's argument to the high dimensions.
The main goal of this article is to provide an
alternate approach and generalize the above result of Perelman  to
the dimensions greater than 3. Instead of
assuming the uniform bound on curvature,  we
only need very mild growth control on the curvature. Maybe more importantly we do
not  assume that the gradient shrinking soliton is
$\kappa$-non-collapsed, as required by the above mentioned result of
Perelman. For the high dimensional case, our method gives a classification of
 gradient shrinking solitons which are locally conformally flat. The following
 is a consequence of our results.

 \begin{theorem}\label{main1} Let $(M^n, g)$  be a
  gradient shrinking
soliton whose Ricci curvature is nonnegative. If $n\ge 4$ we assume that $(M, g)$ is
 locally conformally flat.
Assume further that
 \begin{equation}\label{eq0}
 |R_{ijkl}|(x)\le \exp(a(r(x)+1))
 \end{equation}
 for some $a>0$, where $r(x)$ is the distance function to a fixed point on the manifold.
 Then its universal cover is either $\R^n$,
$\Sph^n$ or $\Sph^{n-1}\times \R$. In the case that $M$ is compact,
the assumptions that the Ricci curvature is nonnegative  and the
growth condition (\ref{eq0}) are not needed.

In particular,
if $(M^n, g)$ has positive Ricci curvature it must be compact.
 \end{theorem}

 As a corollary we have a more general result than
 Theorem \ref{perelman1}.

\begin{corollary} Let $(M^3, g)$ be a
 gradient shrinking soliton whose Ricci curvature is positive and satisfying (\ref{eq0}).
 Then $M$ must be compact.
\end{corollary}

Some new invariant cones,
which bounds the Weyl curvature by the scalar curvature, have been  discovered in
\cite{BW2}  very recently. This might be related to our result.

The rotationally symmetric gradient shrinking solitons has been studied in \cite{Brett}. It can be easily checked that the rotational symmetric manifolds have vanishing Weyl curvature. Hence our result gives a self-contained
classification on rotationally symmetric gradient shrinking solitons. (The proof in \cite{Brett} appealed the strong result of B\"ohm-Wilking. However it does not require the curvature growth condition for the noncompact case.)

\bigskip
 {\it Acknowledgement}.  The first author would like to thank T. Ilmanen
 for his interests  and his hospitality during the first author's visit at ETH this summer.

\section{Preliminaries}

Recall that $(M, g)$ is a gradient shrinking soliton if there exists
a function $f$ such that its Hessian $f_{ij}$ satisfying
$$
R_{ij}+f_{ij} -\frac{1}{2}g_{ij}=0.
$$
As shown in \cite{CLN}, Theorem 4.1, there exists a family of
metrics $g(t)$, a solution to Ricci flow with the property that
$g(0)=g$ and a family of diffeomorphisms $\phi(t)$, which is
generated by the vector field  $X=\frac{1}{\tau} \nabla f$, such
that $\phi(0)=\id$ and $g(t)=\tau(t) \phi^*(t) g$ with
$\tau(t)=1-t$, as well as $f(t)=\phi^*(t) f$. The following can be
checked some straight forward computations \cite{CLN}.

\begin{lemma} For $\tau>0$,
\begin{eqnarray}
\frac{\partial}{\partial \tau} |\Ric|^2 &=& -\frac{2}{\tau} |\Ric|^2
-\langle \nabla |\Ric|^2, \nabla f\rangle, \label{help1}\\
\frac{\partial}{\partial \tau} S^2 &=& -\frac{2}{\tau} S^2 -\langle
\nabla S^2, \nabla f\rangle. \label{help2}
\end{eqnarray}
Here $S$ is the scalar curvature.
\end{lemma}
This particularly holds at $t=0$ (namely $\tau=1$). A direct
consequence is that
\begin{equation} \label{help3}
\frac{\partial}{\partial t} \left(\frac{|\Ric|^2}{S^2}\right)
=\langle \nabla \left(\frac{|\Ric|^2}{S^2}\right), \nabla f\rangle.
\end{equation}

We shall also need the following results. First we need Proposition
1.1 of \cite{N} to bound the scalar curvature from below.

\begin{proposition}\label{ni1}
Assume that $(M, g)$ is a non-flat gradient shrinking soliton.
Assume that it has nonnegative Ricci curvature. Then there exists
$\delta=\delta(M)>0$ such that $S\ge \delta$.
\end{proposition}

The following result on the bound of $f$ as well as its gradient is
implicit in the argument of \cite{P2} (see also the proof of
Proposition 1.1 of \cite{N}).

\begin{lemma} \label{per-e1} Assume the same assumption as in Proposition \ref{ni1}.
There exist constants $B=B(M,f)$,  $C=C(M,f)>0$ such that
\begin{equation}\label{help4}
f(x)\ge \frac{1}{8} r^2(x) -C
\end{equation}
and
\begin{equation}\label{help5}
f(x)\le 2 r^2(x),\quad \quad |\nabla f|(x) \le 4r(x)
\end{equation}
for $r(x)\ge B$.
 Here $r(x)$ is the distance function to some fixed point $o\in
M$ with respect to $g(0)$ metric.
\end{lemma}

Under the assumption that $(M, g)$ has nonnegative Ricci curvature,
it is also easy, from the soliton equation, to have that (e.g. from
 the proof of Proposition 1.1 in \cite{N})
\begin{equation}\label{help-scalarg}
|\nabla S|^2 \le 4S^2|\nabla f|^2.
\end{equation}
Using the soliton equation and the assumption that $R_{ij}\ge 0$ we
also have that
\begin{equation}\label{help-fhess}
|f_{ij}|^2 \le \max\{ \frac{n}{2}, S^2\}.
\end{equation}
 We need these inequalities to justify the finiteness of
some integrals. Most importantly recall the following local
derivative estimates of Shi (cf. Theorem 13.1 of \cite{H93}).

\begin{theorem}\label{shi-local}
For any $\alpha>0$ there exists a constant $C(n, K, r, \alpha)$ such
that if $(M, g(t))$ is a solution to Ricci flow with $t\in [0,
t_1]$, $0< t_1\le \frac{\alpha}{K}$, $p\in M$ and
$$
|R_{ijkl}|(x, t)\le K
$$
for all $x\in B_{g(0)}(p, r), \quad t\in [0, t_1]$, then
$$
|\nabla_s R_{ijkl}|(y, t)\le \frac{C(n, \sqrt{K}r,
\alpha)K}{\sqrt{t}}
$$
for all $y\in B_{g(0)}(p, \frac{r}{2})$ and $t\in (0, t_1]$.
Moreover, for the above $(y, t)$
$$
|\nabla^m R_{ijkl}|(y, t)\le \frac{C(n, m, K,
r,\alpha)}{t^{\frac{m}{2}}}.
$$

\end{theorem}

\section{Three dimensional case}

We first give a different proof to Perelman's theorem mentioned in
the introduction.  In fact what we prove is a more general result
since we  assume neither that gradient shrinking soliton is
$\kappa$-noncollapsed nor that  the curvature is uniformly bounded.
Most argument of the proof can also be used in dimensions $n\ge 4$.

We assume that the Ricci curvature  satisfies that for any
$\epsilon>0$, there exists $\beta(\epsilon)>0$ such that
\begin{equation}\label{main-assum0}
|\Ric|(y,t) \le \exp(\epsilon r^2(x)+\beta(\epsilon))
\end{equation}
for all $y\in B_{g(-\frac{1}{2})}(x, \frac{r(x)}{2})$ and $t\in
[-\frac{1}{2}, 0]$. Here $r(x)$ is the distance function to some
fixed point $o\in M$ with respect to the metric $g(0)$. Notice that
(\ref{main-assum0}) can be easily verified if we assume that
$|\Ric|$ is uniformly bounded at $t=0$. The main purpose of this
section is to show the following result.

\begin{theorem} \label{main3d} Let $(M^3, g)$ be a complete gradient shrinking
soliton with the positive sectional curvature and the Ricci
curvature satisfies (\ref{main-assum0}). Then $M$ must be the
quotient of $\Sph^3$.
\end{theorem}

Note that we do not need to assume $(M, g)$ is
$\kappa$-non-collapsed. The proof also concludes that
$M=\Sph^3/\Gamma$ directly without appealing to Hamilton's result.
 In the later section  we in fact  directly obtain a
classification of solitons under the assumption that $\Ric\ge 0$.

First we recall a result of Hamilton. In \cite{H82}, the following
result was proved for solutions to Ricci flow on three manifold $M$.

\begin{proposition}\label{h82-cal1}
\begin{equation}\label{heat1}
\heat
\left(\frac{|\Ric|^2}{S^2}\right)=-\frac{2}{S^4}\left|S\nabla_p
R_{ij} -\nabla_p S R_{ij}\right|^2 -\frac{P}{S^3} +\langle \nabla
\left( \frac{|\Ric|^2}{S^2}\right), \nabla \log S^2\rangle,
\end{equation}
where
$$
P=\frac{1}{2}\left((\mu+\nu-\lambda)^2
(\mu-\nu)^2+(\lambda+\nu-\mu)^2(\lambda-\nu)^2
+(\lambda+\mu-\nu)^2(\lambda -\mu)^2\right)
$$
with $\mu$, $\nu$ and $\lambda$ are eigenvalues of $\Ric$.
\end{proposition}

If $(M^3, g)$ is gradient shrinking soliton, combining the
discussions above we have that at $t=0$,
\begin{eqnarray}\label{key-eq1}
0&=&\Delta \left(\frac{|\Ric|^2}{S^2}\right)-\langle \nabla
\left(\frac{|\Ric|^2}{S^2}\right), \nabla
f\rangle-\frac{2}{S^4}\left|S\nabla_p R_{ij} -\nabla_p S
R_{ij}\right|^2 \\
&\, &-\frac{P}{S^3} +\langle \nabla \left(
\frac{|\Ric|^2}{S^2}\right), \nabla \log S^2\rangle.\nonumber
\end{eqnarray}
Now multiply $|\Ric|^2 e^{-f}$ on the both sides of the above
equation  then integrate by parts. Here we have assumed that all
integrals involved are finite and the integration by parts can be
performed, which we justify later.

\begin{eqnarray*}
0&=& \int_M -\langle \nabla \left(\frac{|\Ric|^2}{S^2}\right),
\nabla |\Ric|^2 \rangle e^{-f}-\frac{2|\Ric|^2}{S^4}\left|S\nabla_p
R_{ij} -\nabla_p S
R_{ij}\right|^2  e^{-f}\\
&\, &\int_M -\frac{P}{S^3}|\Ric|^2 e^{-f} +\langle \nabla \left(
\frac{|\Ric|^2}{S^2}\right), \nabla \log S^2\rangle |\Ric|^2 e^{-f}.
\end{eqnarray*}
Using that
$$
\nabla \left(\frac{|\Ric|^2}{S^2}\right)=\frac{\nabla
|\Ric|^2}{S^2}-\frac{\nabla S^2}{S^4}|\Ric|^2
$$
we have that
\begin{eqnarray*}
&\quad&\int_M -\langle \nabla \left(\frac{|\Ric|^2}{S^2}\right),
\nabla |\Ric|^2 \rangle e^{-f}+\langle \nabla \left(
\frac{|\Ric|^2}{S^2}\right), \nabla \log S^2\rangle |\Ric|^2
e^{-f}\\
&= &-\int_M \left|\nabla\left(\frac{|\Ric|^2}{S^2}\right)\right|^2
S^2 e^{-f}.
\end{eqnarray*}
Hence we have that
\begin{eqnarray*}
0&=& \int_M -\left|\nabla\left(\frac{|\Ric|^2}{S^2}\right)\right|^2
S^2 e^{-f}-\frac{2|\Ric|^2}{S^4}\left|S\nabla_p R_{ij} -\nabla_p S
R_{ij}\right|^2  e^{-f}\\
&\, &\int_M -\frac{P}{S^3}|\Ric|^2 e^{-f}.
\end{eqnarray*}
In particular, $\frac{|\Ric|^2}{S^2}$ is a constant,
\begin{equation}\label{con1}
S\nabla_p R_{ij} -\nabla_p S R_{ij}=0
\end{equation}
and $P=0$. If we choose a orthornormal frame such $R_{ij}$ is
diagonal, the equality (\ref{con1}) implies that
\begin{eqnarray}
S \nabla_p R_{jj} &=&\nabla_p S R_{jj} \label{d-1}\\
S\nabla_p R_{ij}&=&0, \mbox{ for } i\ne j \label{d-2}
\end{eqnarray}
Using the second Bianchi identity:
$$
\frac{1}{2}\nabla_i S =\sum_p \nabla _p R_{ip} =\nabla_i R_{ii}
$$
we have that
$$
\frac{1}{2}S\nabla_i S =S \nabla_j R_{jj}=\nabla_i S R_{ii}.
$$
On the other hand, $P=0$ implies that
$R_{11}=R_{22}=R_{33}=\frac{1}{3}S$. We thus have that
$$
\frac{1}{2}S\nabla_i S=(\nabla_i S) \frac{S}{3}
$$
which implies that $S$ is a constant. Then (\ref{d-1}) and
(\ref{d-2}) implies that $\nabla_p R_{ij}=0$ for any $p, i, j$. This
implies that $M$ is a compact locally symmetric space with positive
curvature. The claim then follows from classical known results.

Now with the help of  Proposition \ref{ni1} and Lemma \ref{per-e1}
we now  justify the finiteness of the integrals involved and the
integration by parts.

First note that if we assume that $\sup_{x\in M}|R_{ijkl}|(x)\le C$
for some $C>0$, namely the curvature is bounded, invoking the
Bernstein-Bando-Shi type derivative estimates (cf. \cite{CK},
Theorem 7.1), we have that $|\nabla^m R_{ijkl}|$ are uniformly
bounded on $M$. Hence all the integrals involved are finite which
then implies, via cut-off function argument, that the integrations
by parts are completely legal,  in view of the fast decay of
$e^{-f}$ ensured by lemma \ref{per-e1} and the lower bound of $S$
provided by Proposition \ref{ni1}.

For the general case, notice first that the assumption on $|\Ric|$
is equivalent to the same  assumption on $|R_{ijkl}|$ (with some
factor of absolute constant). Hence we have that for any
$\epsilon>0$, there exists $\beta(\epsilon)>0$ such that
\begin{equation}\label{main-assum1}
|R_{ijkl}|(y,t) \le \exp(\epsilon r^2(x)+\beta(\epsilon))
\end{equation}
for all $y\in B_{g(-\frac{1}{2})}(x, \frac{r(x)}{2})$ and $t\in
[-\frac{1}{2}, 0]$.

Below we estimate $\Delta \left(\frac{|\Ric|^2}{S^2}\right)
|\Ric|^2$. The others are similar. Applying the local derivative
estimate of Shi (cf. Theorem 13.1 of \cite{H93}) we have that
\begin{eqnarray*}
|\nabla_p R_{ijkl}|(x, 0) &\le& C_1 \exp(\frac{3}{2}\epsilon r^2(x)
+\beta_1(\epsilon))\\
|\nabla_p\nabla_q R_{ijkl}|(x, 0)&\le & C_2 \exp(\frac{9}{4}\epsilon
r^2(x)+\beta_2(\epsilon)).
\end{eqnarray*}
Direct computation shows that
\begin{eqnarray*}
\Delta\left(\frac{|\Ric|^2}{S^2}\right)&=& \frac{\Delta
|\Ric|^2}{S^2}-2\frac{\langle \nabla |\Ric|^2, \nabla \log
S^2\rangle}{S^2}+2|\Ric|^2 \frac{|\nabla \log S^2|^2}{S^2}\\
&\quad & -\frac{\Delta S^2}{S^4}|\Ric|^2.
\end{eqnarray*}
At $t=0$ there exists absolute constants $C_i, \, i=3,4,5$  and
$\beta_3(\epsilon)$ depending only on $\beta_1$ and $\beta_2$ such
that for $r(x)>>1$,
\begin{eqnarray*}
I&=&\left(\left|\frac{\Delta |\Ric|^2}{S^2}\right||\Ric|^2\right)(x,
0)\le
\frac{C_3}{\delta^2} \exp(6\epsilon r^2(x)+\beta_3(\epsilon)),\\
II&=&\left(\left|\frac{\langle \nabla |\Ric|^2, \nabla \log
S^2\rangle}{S^2}\right||\Ric|^2 +|\Ric|^4 \frac{|\nabla \log
S^2|^2}{S^2}\right)(x,0)\\
&\quad& \le \frac{C_4}{\delta^2}\exp(6\epsilon
r^2(x)+\beta_3(\epsilon)),\\
III&=&\left(\left|\frac{\Delta S^2}{S^4}\right||\Ric|^4\right)(x,
0)\le\frac{C_5}{\delta^2}\exp(6\epsilon r^2(x)+\beta_3(\epsilon)).
\end{eqnarray*}
In the last one we have used the estimates (\ref{help5}),
(\ref{help-scalarg}), (\ref{help-fhess}), as well as
$$
\nabla_i S =2R_{ij} f_j
$$
which then implies
$$
\Delta S =\langle \nabla S, \nabla f\rangle +2R_{ij}f_{ij}\le
|\nabla S||\nabla f|+2S|\sum_{ij}f^2_{ij}|^{1/2}.
$$
Putting the above estimates together with (\ref{help4}) we conclude
that at $t=0$,
$$
\int_M \left|\Delta\left(\frac{|\Ric|^2}{S^2}\right)\right||\Ric|^2
e^{-f}\, d\mu_0<\infty.
$$
Similarly one can establish the finiteness of other integrals
involved. Once we have the the finiteness of the  integration, the
integrations by parts can be checked by approximation via the
cut-off functions. This is somewhat standard we hence omit the
details.

\begin{remark} An argument similar to the one used here was  originated by
Huisken in  his  classification of mean convex shrinking solitons of
mean curvature flow in $\R^{n+1}$ \cite{Hu2}.
\end{remark}

\section{High dimension-preliminaries}
Most results in this section are either known (cf. \cite{Hu, H86})
or can be derived easily from the known ones in the literature.  We
include them here for the completeness. We also adapt them into the
form needed by us.

Recall the evolution formula of the curvature under the Ricci flow
\cite{H82}:
\begin{eqnarray*}
\heat R_{ijkl}&=& 2(\Rc^2+\Rc^{\#})_{ijkl} \\
&\,
&-\left(R_{ip}R_{pjkl}+R_{jp}R_{ipkl}+R_{kp}R_{ijpl}+R_{lp}R_{ijkp}\right)
\end{eqnarray*}
where $ Q(\Rc)=\Rc^2+\Rc^{\#} $ is defined via the Lie algebra
structure of $\wedge^2(n)$, which can be identified with the Lie
algebra of $O(n)$. The below is a brief explanation.

Let $(E, g)$ be a Euclidean space with metric $g$. We can make the
following identifications: $\otimes^2E$, the  tensor space, can be
identified  with $GL(n, \R)$, the linear transformations on $E$ (for
any $x\otimes y\in \otimes^2E$, $x\otimes y(z)=\langle y, z\rangle
x$ is the corresponding element of $GL(n, \R)$); under this
identification, the space symmetric two tensors $S^2E$ corresponds
to  the symmetric transformations $S^2(E)$; $\wedge^2E$ can be
identified with $so(n)$ ($e_i\wedge e_j=e_i\otimes e_j -e_j\otimes
e_i$ is identified with $E_{ij}$ with $1$ at $(i, j)$-th position
and $-1$ at $(j, i)$-th position. The metric on $TM$ extends
naturally to all the related tensor spaces such as $\otimes^2TM$,
$S^2TM$, $\wedge^2TM$. With respect to the previous identification,
the metric on $so(n)$ is given by $\langle A,
B\rangle=-\frac{1}{2}\operatorname{tr}(AB)$
($=\frac{1}{2}\operatorname{tr}(A^tB))$ such that $\{e_i\wedge
e_j\}_{i<j}$ is an orthonormal basis of $\wedge^2 TM$. The
identification also equips $\wedge^2 TM$ with a Lie algebra
structure, which is of fundamental importance in the study of
evolution of curvature operators under  Ricci flow. This  was first
observed by Hamilton \cite{H86}. Let us recall this fact first. For
an orthonormal basis $\phi_{\alpha}$ of $\wedge^2TM$ (say
$\phi_\alpha=e_i\wedge e_j$, which is identified with $E_{ij}$), the
Lie bracket is given by
$$
[\phi_\alpha, \phi_\beta]=c_{\alpha \beta \gamma} \phi_\gamma.
$$
It is easy to check, by simple linear algebra, that
$$
\langle [\phi, \psi],\omega\rangle =-\langle [\omega, \psi],
\phi\rangle.
$$
This immediately implies that $c_{\alpha \beta\gamma}$ is
anti-symmetric. If $A, B\in S^2(\wedge^2TM)$ one can define
$$
(A\#B)_{\alpha \beta}=\frac{1}{2}c_{\alpha \gamma \eta}c_{\beta
\delta \theta}A_{\gamma \delta}B_{\eta \theta}.
$$
It is easy to see that $A\#B$ is symmetric too. Also from the
anti-symmetry of $c_{\alpha\beta\gamma}$
$$
A\#B=B\#A.
$$
The easy computation also shows that
$$
\langle (A\#B)(\phi), \psi\rangle =\frac{1}{2}\sum_{\alpha
\beta}\langle [A(\omega_\alpha), B(\omega_\beta)], \phi\rangle \cdot
\langle [\omega_\alpha, \omega_\beta], \psi \rangle
$$
if $\{\omega_\alpha\}$ is an orthonormal basis.  This particularly
implies that $\operatorname{tr}((A\#B)\cdot C)$ is symmetric in $A,
B, C$ since
\begin{eqnarray*}
\operatorname{tr}((A\#B)\cdot C)&=&\sum_{\gamma}\langle (A\#B)\cdot
C(\omega_\gamma), \omega_\gamma\rangle\\
&=&\frac{1}{2}\sum_{\alpha \beta\gamma}\langle [A(\omega_\alpha),
B(\omega_\beta)], C(\omega_\gamma)\rangle \langle [\omega_\alpha,
\omega_\beta], \omega_\gamma\rangle.
\end{eqnarray*}
Now define
$$
\tri(A, B, C)=\tr((AB+BA+2A\#B)C)
$$
which is symmetric in all variables. If we write
$$
\Rc(e_i\wedge e_j)=\frac{1}{2} \sum_{k, l}R_{ijkl}e_k\wedge e_l
$$
we would have that
$$
|R_{ijkl}|^2 =4\langle \Rc, \Rc\rangle.
$$
We denote $\tri(\Rc)=\tri(\Rc, \Rc, \Rc)=\langle 2(\Rc^2 +\Rc^{\#}),
\Rc\rangle$ and $Q(\Rc)=\Rc^2+\Rc^{\#}$.

The curvature operator $\Rc$ has an orthogonal splitting, with
respect irreducible $O(n)$ representation, into the trace part
$\Rc_{\I}=\frac{S}{n(n-1)}\I$,  the traceless Ricci part
$\Rc_{\Ric_0}=\frac{2}{n-2}\Ric_0\wedge \id$, where $\Ric_0$ denotes the traceless part of
the Ricci curvature, and  the Weyl
curvature $\Rc_W$ (cf. \cite{BW}). We denote the three subspaces  by $\langle \I\rangle$, $\langle \Ric_0\rangle$ and $\langle W\rangle$ respectively.  Equipped with the above notations
we have that
\begin{lemma}
\begin{equation}\label{co1}
\heat |R_{ijkl}|^2 =8\tri (\Rc) -2|\nabla_p R_{ijkl}|^2.
\end{equation}
\end{lemma}
Direct calculation then yields the following
\begin{proposition}\label{co2}
Assume that $S>0$. Then
\begin{eqnarray}\label{gauss-pre1}
&\, &\heat \left(\frac{|R_{ijkl}|^2}{S^2}\right)=
\frac{4}{S^3}\left(2\tri(\Rc) S -\sigma^2 |R_{ijkl}|^2\right)\\
&\quad& \quad  -\frac{2}{S^4}\left|S\nabla_p R_{ijkl}-\nabla_p S
R_{ijkl}\right|^2 +\langle \nabla
\left(\frac{|R_{ijkl}|^2}{S^2}\right), \nabla \log S^2\rangle,
\nonumber
\end{eqnarray}
where $\sigma^2 =|\Ric|^2$.
\end{proposition}

Note that Tachibana  \cite{T} proved that (see also \cite{CLN},
pages 267-269), under the assumption that $\Rc\ge 0$,
$$
-2\tri(\Rc) +\Ric(\Rc, \Rc)\ge 0
$$
where $\Ric(\Rc, \Rc)=R_{ip}R_{ijkl}R_{pjkl}$.

In \cite{Hu}, Huisken obtained the following identities.
\begin{eqnarray}\label{hu1}
&\, &(\Rc_{\I})_{ijkl}(Q(\Rc))_{ijkl}=4\langle Q(\Rc),
\Rc_{\I}\rangle=\frac{2}{n(n-1)}S\sigma^2;\\\label{hu2} &\,
&(\Rc_{\Ric_0})_{ijkl}(Q(\Rc))_{ijkl}=
\frac{4}{n(n-1)}S\tilde{\sigma}^2-\frac{8}{(n-2)^2} \lambda_i^3
+\frac{4}{n-2} (\Rc_W)_{ijij}\lambda_i \lambda_j;\\&\, &
(\Rc_{W})_{ijkl}(Q(\Rc))_{ijkl}= 2\tri(\Rc_W)
+\frac{2}{n-2}(\Rc_W)_{ijij}\lambda_i \lambda_j;\label{hu3}
\end{eqnarray}
where $\lambda_i$ are the eigenvalues of $\Ric_0$ and
$\tilde{\sigma}^2=\sum \lambda_i^2$. Below we first show these
equations via the following lemma, which essentially follows from
\cite{BW}. In \cite{Hu}, the result was shown by direct but long
computations which were omitted. With the help of \cite{BW}, the
result can be obtained without much computation. We include the
derivation for the sake of completeness. First we need to following
lemma which has been essentially proved in \cite{BW}.

\begin{lemma}
\begin{equation}\label{bw1}
\Rc+\Rc\# \I =\Ric(\Rc)\wedge \id.
\end{equation}
Hence for any $\Rc_1, \Rc_2\in S_B(\wedge^2(n))$, let
$$
B(\Rc_1, \Rc_2)=\Rc_1\Rc_2 +\Rc_2 \Rc_1 +2\Rc_1\# \Rc_2.
$$
Let $\Rc^i_I\in \langle \I\rangle, \ \Rc_0\in \langle
\Ric_0\rangle$, $W_i, W\in \langle W\rangle$ ($i=1,2$). Then the
following hold
\begin{eqnarray}
B(\Rc_{\I}, W)&=&0, \nonumber\\B(\Rc^1_{\I}, \Rc^2_{\I})&\in& \langle \I\rangle \nonumber\\
B(W_1, W_2)&\in& \langle W \rangle \nonumber\\B(\Rc_{\I},
\Rc_0)&\in& \langle
\Ric_0\rangle\nonumber\\
B(\Rc_0, W)&\in& \langle \Ric_0\rangle\nonumber \\
\frac{1}{2}B(\Rc_0, \Rc_0)&=& \frac{1}{n-2}\Ric_0\wedge \Ric_0
-\frac{2}{(n-2)^2}(\Ric_0^2)_0\wedge \id +\frac{\tilde
\sigma^2}{n(n-2)}\I. \label{eq1}
\end{eqnarray}
Moreover
\begin{equation}\label{eq2}
\Ric_0\wedge \Ric_0=-\frac{\tilde{\sigma}^2}{n(n-1)}\I
-\frac{2}{n-2}(\Ric_0^2)_0 \wedge \id +\left(\Ric_0\wedge
\Ric_0\right)_W.
\end{equation}
\end{lemma}

Equipped with the above lemma we have that
\begin{eqnarray*}
\tri(\Rc, \Rc, \Rc_{\I}) &=& \tri(\Rc, \Rc_{\I}, \Rc)\\
&=& \frac{2S}{n(n-1)}\langle \Ric\wedge \id, \Rc\rangle\\
&=&\frac{2S}{n(n-1)}\langle \frac{S}{n}\id\wedge \id+\Ric_0\wedge
\id, \frac{S}{n(n-1)} \id\wedge \id +\frac{2}{n-2}\Ric_0\wedge \id
\rangle
\end{eqnarray*}
If we let $\bar{\lambda}=\frac{S}{n}$, we have that
\begin{eqnarray*}
&\, & \langle \frac{S}{n}\id\wedge \id+\Ric_0\wedge \id,
\frac{S}{n(n-1)} \id\wedge \id +\frac{2}{n-2}\Ric_0\wedge \id
\rangle\\
&=&\langle \bar{\lambda}\id\wedge \id+\Ric_0\wedge \id,
\frac{\bar{\lambda}}{n-1} \id\wedge \id +\frac{2}{n-2}\Ric_0\wedge
\id \rangle=\frac{n}{2}\bar{\lambda}^2 +\frac{1}{2}\sum
\lambda_i^2\\
&=& \frac{1}{2}\sigma^2.
\end{eqnarray*}
This proves (\ref{hu1}). For (\ref{hu2}), let $\Rc_0=\Rc_{\Ric_0}$.
We need to compute $\tri(\Rc, \Rc, \Rc_0)$.  Using the symmetry
\begin{eqnarray*}
\tri(\Rc, \Rc, \Rc_0)&=& \tri(\Rc, \Rc_0, \Rc)\\
&=& \langle B(\Rc_{\I}, \Rc_0), \Rc\rangle +\langle B(\Rc_0, \Rc_0),
\Rc\rangle
+\langle B(\Rc_W, \Rc_0), \Rc\rangle\\
&=&\langle B(\Rc_{\I}, \Rc_0), \Rc_0\rangle+\langle B(\Rc_0, \Rc_0),
\Rc_{\I}+\Rc_0+\Rc_W\rangle +\langle B(\Rc_W, \Rc_0), \Rc_0\rangle\\
&=& 2\tri(\Rc_0, \Rc_0, \Rc_{\I}) +2\tri(\Rc_0, \Rc_0,
\Rc_W)+\tri(\Rc_0, \Rc_0, \Rc_0).
\end{eqnarray*}
Using (\ref{eq1}) and (\ref{eq2}) we have that
$$
\tri(\Rc_0, \Rc_0, \Rc_{\I})=\frac{1}{n(n-1)}\tilde{\sigma}^2S
$$
In a similar way,
$$
\tri(\Rc_0, \Rc_0, \Rc_0)=-\frac{4}{(n-2)^2}\sum\lambda_i^3
$$
and
$$
2\tri(\Rc_0, \Rc_0, \Rc_W)=\frac{2}{n-2}(R_W)_{ijij}\lambda_i
\lambda_j.
$$
The above three give (\ref{hu2}). For (\ref{hu3}), notice that
$$
\tri(\Rc_W, \Rc, \Rc)=\tri(\Rc_0, \Rc_W, \Rc_0)+\tri(\Rc_W, \Rc_W,
\Rc_W).
$$
Then the claimed equality follows from the above computation on
$\tri(\Rc_0, \Rc_0, \Rc_W)$.

Finally one can arrive at the following formula.
\begin{eqnarray}
2\tri(\Rc) S -\sigma^2 |R_{ijkl}|^2&=& -4|\Rc_W|^2 \sigma^2+2S
\tri(\Rc_W,
\Rc_W, \Rc_W)\nonumber \\
&\, &-\frac{4}{n(n-1)(n-2)}S^2
\tilde{\sigma}^2-\frac{4}{n-2}\tilde{\sigma}^4 \label{hu-main1}\\
&\, & -\frac{8}{(n-2)^2}S\sum \lambda_i^3 +\frac{6}{n-2} S
(\Rc_{W})_{ijij} \lambda_i\lambda_j. \nonumber
\end{eqnarray}
This follows from (\ref{hu1})-(\ref{hu3}) along with the observation
that
\begin{eqnarray*}
|\Rc|^2&=&|\Rc_{\I}|^2+|\Rc_{\Ric_0}|^2+|\Rc_W|^2\\
&=& \frac{S^2}{2n(n-1)}+\frac{1}{n-2}\sum \lambda_j^2+|\Rc_W|^2
\end{eqnarray*}
and
$$
\sigma^2=\frac{S^2}{n}+\tilde{\sigma}^2.
$$
Hence
$$
4\sigma^2|\Rc|^2=\frac{2S^2}{n(n-1)}\sigma^2 +\frac{4
S^2\tilde{\sigma}^2}{(n-2)n}+\frac{4}{n-2}\tilde{\sigma}^4+4|\Rc_W|^2\sigma^2.
$$
In the case that $\Rc_W=0$, which is automatical if $n=3$ and
amounts to that $(M, g)$ is locally conformally flat if $n\ge 4$, we have
that
\begin{equation}\label{sim1}
2\tri(\Rc) S -\sigma^2 |R_{ijkl}|^2=-\frac{4}{n(n-1)(n-2)}S^2
\tilde{\sigma}^2-\frac{4}{n-2}\tilde \sigma^4-\frac{8}{(n-2)^2}S\sum
\lambda_i^3.
\end{equation}

 Similarly, using that
$$
\heat R_{ik} =2R_{ijkl}R_{jl}-2R_{il}R_{lk}
$$
 we also have the high dimensional analogue of
Proposition \ref{h82-cal1}.

\begin{proposition}\label{co3} Assume that $S>0$. Then
\begin{eqnarray}
&\, &\heat \left(\frac{\sigma^2}{S^2}\right)=
\frac{4}{S^3}\left(SR_{ijkl}R_{jl}R_{ik}-\sigma^4\right)\\
&\quad& \quad -\frac{2}{S^4}\left|S\nabla_p R_{ij} -\nabla_p S
R_{ij}\right|^2 +\langle \nabla \left(\frac{\sigma^2}{S^2}\right),
\nabla \log S^2\rangle. \nonumber
\end{eqnarray}
\end{proposition}
In the case $\dim(M)=3$ the above recovers Hamilton's computation
Proposition \ref{h82-cal1}.

\section{High dimension-locally conformally flat}
We first prove the following algebraic result.

\begin{proposition} \label{algebra-help} Assume that $(M, g)$ is locally conformally flat. Let $\sigma, \tilde{\sigma}, \lambda_i$ be as in the last section. Then
$$
2\tri(\Rc) S -\sigma^2
|R_{ijkl}|^2=-\frac{4}{n-2}\left(\frac{1}{n(n-1)}S^2
\tilde{\sigma}^2+\tilde{\sigma}^4+\frac{2}{n-2}S\sum
\lambda_i^3\right)\le 0.
$$
If the equality holds, then either

(i) $\lambda_i=0$ for all $1\le i \le n$, or

(ii) there exists $a>0$ such that
\begin{eqnarray*}
\lambda_l &=&\frac{1}{\sqrt{n(n-1)}}a, \quad \mbox{ for } 1\le l\le
n-1;\\
\lambda_n &=&-\sqrt{\frac{n-1}{n}}a
\end{eqnarray*}
and $S=\sqrt{n(n-1)}a$.
\end{proposition}
\begin{proof}  Let
$$
f(S, \lambda_1, \cdot\cdot\cdot, \lambda_n)=\frac{1}{n(n-1)}S^2
\sum\lambda_i^2+\frac{2}{n-2}S\sum
\lambda_i^3+\left(\sum\lambda_i^2\right)^2.
$$
The goal is to show that $f\ge 0$ under the constraint that
$\sum\lambda_i =0$ and analyze the equality case. Since it is
homogenous we can consider the extremal values of $f$ under the
further constraint $\sum \lambda_i^2=1$.  Viewing $f$ as a quadratic form in $S$, the result
follows, by elementary consideration, if we  show that
$$
\left(\sum \lambda_i^3\right)^2 \le \frac{(n-2)^2}{n(n-1)}
$$
under the constraints $\sum \lambda_i=0$ and $\sum \lambda_i^2=1$.
Let $g=\sum_i \lambda_i^3$. By the Lagrangian multipliers methods,
at the critical points we have that
\begin{eqnarray*}
3\lambda_j^2-\lambda -2\mu \lambda_j&=&0, \quad \mbox{ for  } 1\le j\le n,\\
\sum \lambda_i &=&0,\\
\sum \lambda_i^2&=& 1.
\end{eqnarray*}
This implies that $\lambda=\frac{3}{n}$ and
$$
\lambda_j = \frac{\mu+\epsilon_j \sqrt{\mu^2+\frac{9}{n}}}{3}
$$
with $\epsilon_j \in \{ -1, 1\}$. We shall compute all  possible
values of $\lambda_j$. We shall divide into two cases.

{\it Case 1} : $n=2k$. Let $\epsilon =\sum_j \epsilon_j$ which takes value
in $\{-2k, -2(k-1), \cdot \cdot \cdot, -2,0,2, \cdot \cdot \cdot,
2(k-1), 2k\}$. Since $\sum_j \lambda_j=0$, it is easy to see that
$\epsilon$ can not take the value $2k$ or $-2k$. If $\epsilon =0$ we
have that $\mu=0$, which then implies that, after permutation of the
index  $\lambda_j =\frac{1}{\sqrt{n}}$ for $1\le j \le k$ and
$\lambda_j =-\frac{1}{\sqrt{n}}$ for $k\le j \le 2k$. In this case
$g=0$.

In general assume that $\epsilon =2(k-i)$ for some $1\le i\le k$. We
shall consider only the case $1\le i \le k-1$ since the rest is
symmetric to it. Without the loss of the generality we may assume
that $\epsilon_j =1$ for $1\le j\le 2k-i$ and $\epsilon_j =-1$ for
$2k-i\le j \le 2k$. In this case
\begin{eqnarray*}
\mu&=&-\frac{3(k-i)}{\sqrt{(2k-i)2k i}},\\
\lambda_l &=&\sqrt{\frac{i}{2k(2k-i)}}, \quad  \mbox{ if } 1\le l \le 2k-i, \\
\lambda_l&=& -\sqrt{\frac{2k-i}{2k i}}, \quad \mbox{ if } 2k-i< l\le
2k.
\end{eqnarray*}
This implies that
$$
g=-\frac{n-2i}{\sqrt{(n-i)i}\sqrt{n}}.
$$
Noticing that $\frac{(n-2i)^2}{(n-i)i}$ is monotone decreasing in
$i$, we can conclude that $ g\ge -\frac{n-2}{\sqrt{n(n-1)}}$.
Symmetrically, for $\epsilon =-2(k-i)$ we can find
$g=\frac{n-2i}{\sqrt{(n-i)i}\sqrt{n}}$. Combining them together we
have that
$$
-\frac{n-2}{\sqrt{n(n-1)}}\le g \le \frac{n-2}{\sqrt{n(n-1)}}.
$$
The minimum is achieved when $i=1$, which implies the second part of
the statement in the proposition.

{\it Case 2}: $n=2k+1$. Again due to the fact that $\sum \lambda_i =0$, $\epsilon$
 takes value in $\{ -(2k-1), \cdot\cdot\cdot, -1, 1, \cdot\cdot\cdot, 2k-1\}$.
 Assume that $\epsilon =2(k-i)+1$ for some $1\le i\le 2k$. We shall only consider $1\le i\le k$
 since the other half is symmetric to this case. Now we assume that $\epsilon_j=1$
 for all $1\le j\le 2k-i+1$, and $\epsilon_j=-1$ for $2k-i+2\le j\le 2k+1$. Now we have that
\begin{eqnarray*}
\mu&=& -\frac{3}{2}\cdot \frac{2(k-i)+1}{\sqrt{2k+1}\sqrt{i}\sqrt{2k-i+1}},\\
\lambda_l &=& \frac{\sqrt{j}}{\sqrt{2k-j+1}\sqrt{2k+1}},
\quad \mbox{ for }\, 1\le 1\le l \le 2k-i+1,\\
\lambda_l &=& -\frac{\sqrt{2k-i+1}}{\sqrt{i}\sqrt{2k+1}},
\quad \mbox{ for } \, 2k-i+2\le l\le 2k+1.
\end{eqnarray*}
From this we can compute that
$$
g=-\frac{n-2i}{\sqrt{n}\sqrt{n-i}\sqrt{i}}.
$$
Again by elementary inequality
$$
-\frac{n-2i}{\sqrt{n-i}\sqrt{i}}\ge -\frac{n-2}{\sqrt{n-1}\sqrt{n}}.
$$
Hence we conclude that $g^2 \le \frac{(n-2)^2}{(n-1)n}$. The minimum achieves when $i=1$.

Combining the above two cases, we complete the proof  that $f\ge 0$.
From the above discussion, it is straight forward to check that the
listed cases are the only two when the inequality can achieve the
equality.
\end{proof}

\begin{corollary} Let $(M^n, g)$ ($n\ge 4$) be a locally conformally flat gradient shrinking
soliton whose Ricci curvature is nonnegative satisfying
(\ref{main-assum1}). Then its universal cover is either $\R^n$,
$\Sph^n$ or $\Sph^{n-1}\times \R$. In the case that $M$ is compact,
the assumptions that the Ricci curvature is nonnegative  and the
growth condition (\ref{main-assum1}) are not needed. In particular,
if $(M^n, g)$ has positive Ricci curvature it must be compact.
\end{corollary}

\begin{proof} Notice that $S$ satisfies the equation $\left(\frac{\partial}{\partial t} -\Delta \right)S=2|\Ric|^2$. By the strong maximum principle we may assume that $S>0$, otherwise $M=\R^n$.

Now as in Section 3 we have that
\begin{eqnarray}\label{key-eq51}
0&=&\Delta \left(\frac{|R_{ijkl}|^2}{S^2}\right)-\langle \nabla
\left(\frac{|R_{ijkl}|^2}{S^2}\right), \nabla
f\rangle-\frac{2}{S^4}\left|S\nabla_p R_{ijkl} -\nabla_p S
R_{ijkl}\right|^2 \\
&\, &-\frac{P}{S^3} +\langle \nabla \left(
\frac{|R_{ijkl}|^2}{S^2}\right), \nabla \log S^2\rangle.\nonumber
\end{eqnarray}
Here
$$
P=-4(2\tri(\Rc) S -\sigma^2
|R_{ijkl}|^2),
$$
which is nonnegative by the lemma.
Multiplying $|R_{ijkl}|^2 e^{-f}$ and integrating by parts, which can be justified similarly as in Section 3,  we have that
\begin{eqnarray*}
0&=& \int_M -\left|\nabla\left(\frac{|R_{ijkl}|^2}{S^2}\right)\right|^2
S^2 e^{-f}-\frac{2|R_{ijkl}|^2}{S^4}\left|S\nabla_p R_{ijkl} -\nabla_p S
R_{ijkl}\right|^2  e^{-f}\\
&\, &\int_M -\frac{P}{S^3}|R_{ijkl}|^2 e^{-f}.
\end{eqnarray*}
By the lemma we have that
\begin{equation}\label{conq1}
\nabla_p S R_{ijkl}=S \nabla_p R_{ijkl}
\end{equation}
which implies that
$$
\nabla_p S R_{ik} =S\nabla_p R_{ik}.
$$
Also the argument of Section 3 implies that
$$
2\tri(\Rc) S -\sigma^2 |R_{ijkl}|^2=-2\left(\frac{1}{12}S^2
\tilde{\sigma}^2+\tilde{\sigma}^4+S\sum \lambda_i^3\right)=0
$$
and $\frac{|R_{ijkl}|^2}{S^2}$ is a constant.

 If $\lambda_i=0$, then $R_{ik} =\frac{S}{n}\delta_{ik}$. By
the second Bianchi identity we have that
$$
\frac{1}{2}S\nabla_iS =S\nabla_p R_{ip} =\frac{S}{4}
\delta_{ip}\nabla_pS.
$$
which implies that $\nabla_p S=0$. Then we have $\nabla_p
R_{ijkl}=0$ by (\ref{conq1}).

If the second case happens, by the lemma we have that $R_{ij}=\frac{\delta_{ij}}{n-1}S$
for $1\le i, j\le n-1$ and $R_{nj}=0$ for $1\le j\le n$. The same computation as in $n=3$
shows that $\nabla_p S =0$, hence $\nabla_p R_{ijkl}=0$, which means
that $(M, g)$ is locally symmetric. The conclusion  follows from the fact that
$(M, g)$ is either Einstein or its
Ricci curvature has constant rank $n-1$ and with $n-1$ identical
nonzero eigenvalues.
\end{proof}

\begin{remark}
(1) The compactness part should be compared with the result in  \cite{NW}, where under certain curvature operator pinching condition, the manifold is shown to be compact.

(2) Whether or not the argument here is sufficient to show that {\it any shrinking
gradient soliton with positive curvature operator must be compact} is an interesting question.
 The K\"ahler case has been resolved in \cite{N}. We hope to return to the  remaining
 cases in the future study.
\end{remark}

Since Proposition \ref{algebra-help} also holds when $n=3$, and $\Rc_W=0$ automatically we have the following corollary which generalizes Theorem \ref{perelman1}.

\begin{corollary}  Let $(M^3, g)$  be a  gradient shrinking
soliton whose Ricci curvature is nonnegative satisfying
(\ref{main-assum0}). Then its universal cover is either $\R^3$,
$\Sph^3$ or $\Sph^{2}\times \R$. In the case that $M$ is compact,
the assumptions that the Ricci curvature is nonnegative is not needed. In particular,
if $(M^3, g)$ has positive Ricci curvature it must be compact.
\end{corollary}

\bibliographystyle{amsalpha}

\end{document}